\documentclass{amsart}

\usepackage{amsfonts,epsfig}
\usepackage{latexsym}
\usepackage{amssymb}
\usepackage{amsmath}
\usepackage{amsthm}
\usepackage{graphics}
\usepackage{verbatim}
\usepackage{enumerate}
\usepackage{hhline}
\usepackage[all]{xy}
\usepackage{setspace}
\usepackage{array,booktabs,ctable}
\usepackage[backref]{hyperref}
\usepackage{color}
\definecolor{mylinkcolor}{rgb}{0.8,0,0}
\definecolor{myurlcolor}{rgb}{0,0,0.8}
\definecolor{mycitecolor}{rgb}{0,0,0.8}
\hypersetup{colorlinks=true,urlcolor=myurlcolor,citecolor=mycitecolor,linkcolor=mylinkcolor,linktoc=page,breaklinks=true}
\usepackage{bigstrut}

\newtheorem{defn}{Definition}[section]
\newtheorem{definition}[defn]{Definition}

\newtheorem{thm}[defn]{Theorem}
\newtheorem{theorem}[defn]{Theorem}

\newtheorem{prop}[defn]{Proposition}

\newtheorem{conj}[defn]{Conjecture}

\theoremstyle{definition}

\newtheorem{remark}[defn]{Remark}
\newtheorem{question}[defn]{Question}

\newcommand{\QQ}{\mathbb Q}

\newcommand{\ZZ}{\mathbb Z}

\newcommand{\FF}{\mathbb F}

\newcommand{\PP}{\mathbb P}

\newcommand{\PSL}{\operatorname{PSL}}

\newcommand{\Gal}{\operatorname{Gal}}

\newcommand{\rank}{\operatorname{rank}}

\newcommand{\tor}{\mathrm{tors}}

\begin{document}



\bibliographystyle{plain}
\title[On ranks of elliptic curves with isogenies]{On the ranks of elliptic curves with isogenies}

\author{Harris B. Daniels}
\address{Department of Mathematics and Statistics, Amherst College, Amherst, MA 01002, USA}
\email{hdaniels@amherst.edu}
\urladdr{http://www3.amherst.edu/~hdaniels/}

\author{Hannah Goodwillie}
\address{Department of Mathematics and Statistics, Amherst College, Amherst, MA 01002, USA}
\email{hannah.goodwillie@gmail.com}




\begin{abstract} 
In recent years, the question of whether the ranks of elliptic curves defined over $\QQ$ are unbounded has garnered much attention. One can create refined versions of this question by restricting one's attention to elliptic curves over $\QQ$ with a certain algebraic structure, e.g., with a rational point of a given order. In an attempt to gather data about such questions, we look for examples of elliptic curves over $\QQ$ with an $n$-isogeny and rank as large as possible. To do this, we use existing techniques due to Rogers, Rubin, Silverberg, and Nagao and develop a new technique (based on an observation made by Mazur) that is more computationally feasible when the naive heights of the elliptic curves are large.
\end{abstract}

\maketitle


\section{Introduction}

It is a fundamental theorem in arithmetic geometry that the rational points on an elliptic curve can be given the algebraic structure of a finitely generated abelian group. This was first proven by Mordell \cite{mordell} in 1922 for elliptic curves over $\QQ$ and then vastly generalized by Weil \cite{weil}, who proved in 1929 that the group of rational points on an abelian variety defined over a number field is finitely generated. Thus, if $E/\QQ$ is an elliptic curve, there exists an integer $r\geq0$ and finite abelian group $E(\QQ)_\tor$ such that
\[
E(\QQ)\simeq \ZZ^r\oplus E(\QQ)_\tor.
\]
The integer $r$ is called the {\it rank} of the elliptic curve and is denoted $\rank_\QQ(E)$, while $E(\QQ)_\tor$ is called the {\it torsion subgroup} of $E$. The torsion subgroups of rational elliptic curves (up to isomorphism) are already completely understood and are classified by the following theorem:
\begin{thm}[Mazur \cite{mazur1}]\label{thm-mazur}
Let $E/\QQ$ be an elliptic curve. Then
\[
E(\QQ)_\tor\simeq
\begin{cases}
\ZZ/M\ZZ &\text{with}\ 1\leq M\leq 10\ \text{or}\ M=12,\ \text{or}\\
\ZZ/2\ZZ\oplus \ZZ/2M\ZZ &\text{with}\ 1\leq M\leq 4.
\end{cases}
\]
\end{thm}

With the torsion subgroups of rational elliptic curves completely classified, we turn our attention to the ranks of rational elliptic curves. A first natural question to ask is:
\begin{question}\label{question-bounded}
Let $S = \{ \rank_\QQ(E): E/\QQ\hbox{ is an elliptic curve}\}$. Is the set $S$ bounded above?	
\end{question}

Despite much effort, this question remains unanswered and the mathematical community seems to be split on what to expect the answer to be. Recently in \cite{PPVW}, Park, Poonen, Voight, and Wood have announced a heuristic that suggests that there are only finitely many elliptic curves of rank greater than 21, which would imply a positive answer to the above question. Their heuristic is based on modeling the ranks and Tate-Shafarevich groups simultaneously using alternating integer matrices.

In an attempt to better understand Question \ref{question-bounded}, many people have attempted to generate examples of elliptic curves with rank as large as possible. Much of the history of this endeavor has been catalogued by Dujella on his website \cite{Dujella} and the current record is an elliptic curve with rank at least 28 given by Elkies in \cite{ElkiesRR}.
Elkies was able to find 28 independent rational points of infinite order, showing that this curve has rank at least 28; with more advanced techniques, he showed that it has rank at most 32. This curve alone does not contradict the heuristic presented in \cite{PPVW} since it permits finitely many elliptic curves with rank greater than 21.

As a refinement of this problem, one could ask if the ranks of elliptic curve over $\QQ$ with some added algebraic structure are unbounded. In particular, the ranks of elliptic curves with a given torsion structure have extensively studied. For example, the current rank record for an elliptic curve defined over $\QQ$ with torsion subgroup isomorphic to $\ZZ/2\ZZ\times\ZZ/8\ZZ$ is only 3. Rather than working on the well-studied question of the ranks of elliptic curves with prescribed torsion structures, we instead modify the question once again. Instead of requiring our elliptic curves to have a fixed torsion subgroup, we ask that they have an isogeny of fixed degree.

\begin{definition}
An elliptic curve $E/\QQ$ is said to have a {\it$\QQ$-rational $n$-isogeny} (or an {\it isogeny of degree $n$}) if $E$ has a cyclic subgroup of order $n$ defined over $\overline{\QQ}$ that is stable under the component-wise action of $\Gal(\overline{\QQ}/\QQ)$.
\end{definition}

From this definition we can see that having an $n$-isogeny is a generalization of having a rational point of order $n$ since any elliptic curve with a rational point of order $n$ has a cyclic subgroup of order $n$ that is stable under the action of $\Gal(\overline{\QQ}/\QQ)$.

With the techniques described below, we are able to prove the following theorem:

\begin{theorem}
If $n$ is a positive integer such that there is an elliptic curve over $\QQ$ with a rational $n$-isogeny, then there exists an elliptic curve over $\QQ$ with an $n$-isogeny and rank greater than or equal to 5.
\end{theorem}

\section{Elliptic Curves with Isogenies}

By studying rational points on the classical modular curves $X_0(n)$, Mazur and Kenku were able to classify all of the integers $n$ for which there is an elliptic curve defined over $\QQ$ with an $n$-isogeny. 
\begin{theorem}\label{thm-isog}\cite{mazur1,kenku2,kenku3,kenku4,kenku5}
Let $E/\QQ$ be an elliptic curve with a rational $n$-isogeny. Then
\[
n\leq 19 \text{ or } n \in\{21,25,27,37,43,67,163\}.
\]
\end{theorem}

Further, all of the $\QQ$-rational points on the curves $X_0(n)$ have been fully classified, thus classifying (and parametrizing by $j$-invariant) all of the $\overline{\QQ}$-isomorphism classes of elliptic curves with an isogeny of a given degree. This work was done by Fricke, Kenku, Klein, Kubert, Ligozat, Mazur, and Ogg, among others. The results are spread vastly throughout the literature, but they have been collected into a single set of tables by Lozano-Robledo in \cite[Tables 3 and 4]{ALR1}.

From these tables, we see that the parametrizations come in two different forms depending on the genus of $X_0(n)$: when the genus of $X_0(n)$ is positive, there are finitely many $\overline{\QQ}$-isomorphism classes of elliptic curves with $n$-isogenies, corresponding to a finite list of $j$-invariants; on the other hand, when the genus is zero, there are infinitely many $\overline{\QQ}$-isomorphism classes, with $j$-invariants parametrized as the image of a rational map. In our search for elliptic curves of moderate rank with $n$-isogenies, these two cases will require different approaches.

For some of the values of $n$ given in Theorem~\ref{thm-isog}, namely $2 \leq n \leq 10$ and $n = 12$, there exist elliptic curves with rational points of order $n$. Since an elliptic curve with a rational point of order $n$ necessarily has an $n$-isogeny, the existing rank records for elliptic curves with prescribed torsion structures carry over to our search for curves of moderate rank with $n$-isogenies. Thus, although these curves with rational points of order $n$ are not the only curves with $n$-isogenies for $2 \leq n \leq 10$ and $n = 12$, in order to avoid duplicating previous efforts we focus on the other values of $n$ given in Theorem~\ref{thm-isog}, for which no rank records have yet been established.

\section{When the genus of $X_0(n)$ is positive.}\label{sec-PositiveGenus}

To illustrate what can be done in this case, we start by giving the following definition:
\begin{definition}
Let $C_1/\QQ$ be a smooth projective curve. A twist of $C_1/\QQ$ is a smooth curve $C_2/\QQ$ such that $C_2/\QQ$ is isomorphic to $C_1$ over $\overline{\QQ}$.		
\end{definition}

If the curves $C_1$ and $C_2$ above are elliptic curves, because the isomorphism between $C_1$ and $C_2$ is defined over $\overline{\QQ}$ and \emph{not} necessarily over $\QQ$, it is possible that the set of rational points $C_1(\QQ)$ and $C_2(\QQ)$ are not isomorphic as groups. In particular, an isomorphism defined over $\overline{\QQ}$ does not necessarily send rational points to rational points, and so it is possible that the ranks of $C_1$ and $C_2$ are different.

\begin{definition}
Let $E/\QQ$ be an elliptic curve given by Weierstrass equation $y^2=x^3+Ax+B$. For a non-zero rational $D \in \QQ^{\times}$, the quadratic twist of $E$ by $D$ is the elliptic curve $E^{(D)}/\QQ$ given by $y^2=x^3+AD^2x+BD^3$.
\end{definition}


It is simple to check that if $D$ is not a rational square, then $E$ and $E^{(D)}$ are isomorphic over $\QQ(\sqrt{D})$ and not isomorphic over $\QQ$.

\begin{prop}\label{prop-twist}
Let $E_1/\QQ$ and $E_2/\QQ$ be elliptic curves that are isomorphic over $\overline{\QQ}$. Further, suppose that $j(E_1) \neq 0$ or $1728$. Then, either $E_1$ and $E_2$ are isomorphic over $\QQ$, or $E_2$ is a quadratic twist of $E_1$. 
\end{prop}

\begin{proof}
Let $E_1$ and $E_2$ be elliptic curves given by the Weierstrass equations $y^2 = x^3+A_1x+B_1$ and $y^2 = x^3 + A_2x+B_2$ with $A_1,A_2,B_1,B_2\in\ZZ$. Since $j(E_1) \neq 0$ or $1728$ we know that $A_1$ and $B_1$ are both nonzero integers. Let $\varphi:E_1 \to E_2$ be an isomorphism defined over $\overline{\QQ}$. By Proposition 3.1 in Chapter 3 of \cite{silverman}, we have that $\varphi$ is given by $\varphi(x,y) = (u^2x,u^3y)$ for some $u\in \overline{\QQ}$. Further, this means that $A_2 = u^4 A_1$ and $B_2 = u^6 B_1$. Since $A_1,$ $A_2,$ $B_1$, and $B_2$ are all integers and $A_2$ and $B_2$ can't both be zero, it must be that $u^2$ is a nonzero rational number. Thus, either $\varphi$ is defined over $\QQ$ or it is defined over a quadratic extension of $\QQ$. 
\end{proof}

\begin{remark}
In the case where $j(E_1) = 0$ or $1728$, $E_2$ may be a cubic or quartic twist of $E_1$, rather than a quadratic one, since either $A_1 = A_2 = 0$ or $B_1 = B_2 = 0$. The curves in these two $\overline{\QQ}$-isomorphism classes have complex multiplication and do not appear in the list of curves that we are considering, so there is no need to consider these types of twists any further.
\end{remark}

\begin{prop}
Let $E_1$ and $E_2$ be elliptic curves defined over $\QQ$ such that $E_2$ is a quadratic twist of $E_1$. Then, if $E_1$ has an $n$-isogeny, then $E_2$ also has an $n$-isogeny.
\end{prop}

\begin{proof}
Again, let $E_1$ and $E_2$ be elliptic curves given by the Weierstrass equations $y^2 = x^3+A_1x+B_1$ and $y^2 = x^3+A_2x+B_2$ with $A_1,A_2,B_1,B_2\in\ZZ$. Further, let $P = (x_0,y_0)$ be the generator of a Galois stable cyclic subgroup of order $n$. Since $E_2$ is a quadratic twist of $E_1$, there is an isomorphism $\varphi:E_1\to E_2$ given by $\varphi((x,y)) = (u^2x,u^3y)$ for some $u$ with $u^2\in \QQ$. Because $u^2$ is rational, we know that $u = r\sqrt{D}$ for some squarefree integer $D$ and some $r \in \QQ$. We aim to show that $\langle \varphi(P) \rangle$ is Galois stable. 

Take $\sigma \in \Gal(\overline{\QQ}/\QQ)$. Since $\langle P \rangle$ is Galois stable we know there exists $m_\sigma\in\ZZ$ such that $P^\sigma = m_\sigma P$. Further, by basic Galois theory we know that $\sigma(u) = \pm u$, so 
\begin{align*}
\varphi(P)^\sigma 
&= (u^2x_0,u^3y_0)^\sigma = (\sigma(u^2x_0),\sigma(u^3y_0)) = (u^2\sigma(x_0),\pm u^3\sigma(y_0))\\
&=\varphi(\pm P^\sigma) = \varphi(\pm m_\sigma P) = \pm m_\sigma \varphi(P).
\end{align*}
Thus $\langle \varphi(P) \rangle$ is Galois stable and of order $n$ and $E_2$ has an $n$-isogeny.

\end{proof}

\subsection{Ranks of Quadratic Twists of Ellipitic Curves}

In the case when there are only finitely many $\overline{\QQ}$-isomorphism classes of elliptic curves with an $n$-isogeny the idea is to search for quadratic twists of a fixed representative that have moderate rank. Much effort has been put into the study of the ranks of quadratic twists of elliptic curves and there are many open questions about the distribution of ranks in families of quadratic twists of elliptic curves. Of particular importance is the following conjecture due to Goldfeld.
\begin{conj}
If $E/\QQ$ is an elliptic curve, then
$$
\lim_{X\to \infty} \frac{ \sum_{|d|\leq X,\ d\ \rm{squarefree}} \rank_{\QQ}(E^{(d)})}{\#\{d\in\ZZ: |d|\leq X\hbox{ $d$ \rm{squarefree}}\}}= \frac{1}{2}.
$$
\end{conj}
This conjecture together with the parity conjecture implies that for a fixed elliptic curve $E$ the ranks of quadratic twists of $E$ should be 0 half the time and 1 half the time. This makes finding twists of a fixed elliptic curve with rank greater than 1 a potentially difficult task and so the most obvious brute force method of checking the rank of every possible twist by a squarefree integer $D$ with $|D|$ less than some bound is not a viable approach. The potentially highly non-trivial nature of the task of computing the rank of an elliptic curve, too, makes the brute force method intractable, so we need a more subtle approach: we want a heuristic that will allow us to determine (via a quicker computation) which twists are most likely to have higher rank than the others. We will then perform the full rank computations on these heuristically indicated curves.

\subsection{The Method of Rogers} To do this, we use the method outlined in \cite{rogers}, where Rogers searches for twists of moderate rank of the congruent number elliptic curve given by $y^2 = x^3-x$. The method begins by letting $E/\QQ$ be an elliptic curve given by Weierstrass equation $y^2 = f(x)$ and then writing the elliptic curve $E^{(D)}$ in the nonstandard form $Dy^2 = f(x)$. With $E^{(D)}$ written this way, Gouv\^ea and Mazur observed in \cite{gouvea} that if $r\in\QQ$ and $D_r$ is the unique square-free integer such that $D_r \equiv f(r) \bmod (\QQ^\times)^2$, then the curve $E^{(D_r)}$ has a rational point whose $x$-coordinate equal to $r$.

The idea, suggested to Rogers by Rubin and Silverberg, is to fix a height bound $H$ and see which integers occur most often as $D_r$ for some $r \in \QQ$ with height less than $H$. The more times a given integer $D$ occurs, the more points of low (naive) height are present on the curve $E^{(D)}$. The twists of $E$ that have many such points should tend to have higher rank and so these twists are the ones whose rank it is most worthwhile to compute. While this is not always completely correct,\footnote{An elliptic curve of rank 1 might have many points of small naive height, while an elliptic curve of higher rank might not have any of the same height.} it does give us an efficient way to identify promising candidates in our search for twists of high rank. 

In the table below, we give the degree $n$ of the isogeny each curve has, the $j$-invariant as well as the $a$-invariants that define the curve with smallest conductor in the $\overline{\QQ}$-isomorphism class, the upper bound $H$ we used for the heights of the rational numbers, and the twists of maximal rank that we found. For each degree $n$, the highest rank found for a curve with an $n$-isogeny is highlighted in boldface type. In each case, we first performed a search using Rogers's method in SAGE, then used Magma to compute the Selmer ranks of the top 100 most frequently appearing twists, and finally computed the actual rank and generators of the most promising curves.

\begin{table}[h]
$$
\arraycolsep=1.4pt\def\arraystretch{1.1}
\begin{array}{c|c|c|c|c|c}
	j&n&\hbox{$a$-invariants of $E$} &H&D&\rank_{\QQ}\left(E^{(D)}\right)\\\hline
-11\cdot 131^3 &11 & [1, 1, 1, -30, -76] &  10^4 & -97966 & 4\\\hline
-2^{15} & 11 &[0, 0, 0, -9504, 365904]  & 10^4 & -64259 & 4\\\hline
 -11^2 & 11& [1, 1, 0, -2, -7]  & 10^4 & -1472782 & {\bf 5} \\\hhline{=|=|=|=|=|=}
 -3^3\cdot 5^3&14 &[0, 0, 0, -2835, -71442] & 10^4 &-883554 & 4\\\hline
3^3\cdot 5^3 \cdot 17^3 & 14 & [0, 0, 0, -595, -5586]& 10^4 & -4541835 & {\bf 5}\\\hhline{=|=|=|=|=|=}
-5^2/2 & 15 &[0, 0, 0, -675, -79650] & 10^4 & -1734994 & {\bf 5}\\\hline
-5^2\cdot 241^3/2^3 &15 &[0, 0, 0, -162675, -25254450] & 10^4 &-3163927 & { \bf 5}\\\hline
-5\cdot 29^3/2^5 & 15 &[0, 0, 0, -3915, 113670] & 10^4 & 1951555 & 4\\\hline
5\cdot 211^3/2^15 & 15 & [0, 0, 0, 28485, -838890] & 10^4 & -39947 & 4 \\\hhline{=|=|=|=|=|=}
-17^2\cdot 101^3/2 & 17 & [1, 0, 1, -3041, 64278] & 10^4 & 703 & {\bf 5} \\\hline
-17\cdot 373^3/2^{17} & 17 & [1, 0, 1, -3041, 64278] & 10^4 & 11951 & {\bf 5}\\\hhline{=|=|=|=|=|=}
-2^{15} \cdot 3^3 & 19 & [0, 0, 1, -38, 90] & 10^4 &182766 & {\bf 5}\\\hhline{=|=|=|=|=|=}
-3^2 \cdot 5^6/2^3 & 21 & [0, 0, 0, -75, 262] & 10^4 & 5107035 & {\bf 5}\\\hline
3^3\cdot 5^3/2 & 21 & [0, 0, 0, 45, -18] & 10^4 & -2606345 & {\bf 5} \\\hline
-3^2 \cdot 5^3\cdot 101^3/2^{21} & 21 & [0, 0, 0, -122715, -33611274] & 10^4 &-531894503 & {\bf 5} \\\hline
-3^3\cdot 5^3\cdot 383^3/2^7 & 21 & [0, 0, 0, -17235, 870894] & 10^4 & 4094 & 4 \\\hhline{=|=|=|=|=|=}
-2^{15}\cdot 3\cdot 5^3 & 27 & [0, 0, 0, -480, 4048] & 10^4 & 1058402 & {\bf 5} \\\hhline{=|=|=|=|=|=}
-7\cdot 11^3 & 37 & [1, 1, 1, -8, 6] & 10^4 & 21880474 & {\bf 5} \\\hline
-7\cdot 137^3\cdot 2083^3 & 37 &[0, 0, 0, -269675595, -1704553285050] & 5\cdot 10^3 & -3791 & 4 \\\hhline{=|=|=|=|=|=}
-2^{18}\cdot 3^3\cdot 5^3 & 43 & [0, 0, 0, -13760, 621264] & 10^4 & 18618 & {\bf 5}\\\hhline{=|=|=|=|=|=}
-2^{15}\cdot3^3\cdot 5^3\cdot 11^3 & 67 &[0, 0, 0, -117920, 15585808] & 10^4 & 37630& {\bf 5} \\\hhline{=|=|=|=|=|=}
-2^{18} \cdot 3^3 \cdot 5^3\cdot 23^3\cdot 29^3 & 163 &[0, 0, 0, -34790720, 78984748304] & 10^3 & 1 & {\bf 1}
\end{array}
$$
\caption{Twists in the positive genus case}
\label{table:positive}
\end{table}

In almost all cases, we were able to use $H=10^4$, but in two cases the size of the coefficients of the elliptic curves with smallest conductor in the $\overline{\QQ}$-isomorphism class prevented us from searching all the way out to $H = 10^4$ in a reasonable amount of time and we were instead forced to take $H = 5 \cdot 10^3$ or even $H = 10^3$. The difficulty here arises from the fact that the task of computing the squarefree part $D_{f(x)}$ of $f(x)$ is equivalent in complexity to the task of factoring $f(x)$. Indeed, for $n=163$ this search did not provide any elliptic curves of rank greater than that of the original (i.e., ``untwisted'') curve. We use a new method to deal with this particular case.

\subsection{A New Method for Finding Twists of Moderate Rank} 

Let $E/\QQ$ be an elliptic curve. If $p$ is a prime of good reduction for $E$, then let $N_p = \# E(\FF_p)$
and $a_p = p+1 - N_p$; otherwise, let 
$$a_p=\begin{cases}
1&\hbox{if $E$ has split multiplicative reduction at $p$},\\
-1&\hbox{if $E$ has non-split multiplicative reduction at $p$},\\
0&\hbox{if $E$ has additive reduction at $p$.}
\end{cases}
$$
A consequence of the Sato-Tate conjecture is that the proportion of primes for which $a_p$ is positive is equal to the proportion of primes for which $a_p$ is negative. With this in mind, Mazur in \cite{mazur2} considers the quantity
$$D_E(t) = \sum_{p\leq t}{\rm sgn}(a_p).$$
Given the Sato-Tate conjecture, one might expect that the values of $D_E(t)$ would be close to zero for large values of $t$, yet Mazur found experimentally that for curves of higher rank, $D_E(t)$ is more biased toward being negative. To illustrate this phenomenon, Mazur plots $(t,D_E(t))$ for $t \leq 10^6$ for the elliptic curves with Cremona labels 11A, 37A, 389A, and 5077A. These curves have rank 0, 1, 2, and 3 respectively, and Mazur observes that the larger the rank of the curve, the larger the bias is for $D_E(t)$ to be negative. 

One can make intuitive sense of this phenomenon with the following two observations: firstly, there should in general be about $p+1$ points on an elliptic curve $E/\FF_p$, and secondly, a curve with more rational points ought to have more points, on average, when reduced modulo a prime $p$. Thus the sign of $a_p$ tells us exactly when the reduction curve has more or fewer points than the expected average, and so the more frequently the number of points modulo $p$ exceeds the naive expectation of $p+1$, the more frequently ${\rm sgn}(a_p)$ will be negative and the more $D_E(t)$ will trend in the negative direction.

Looking at figures 2.2 -- 2.5 in \cite{mazur2}, we see that while $D_E(10^6)$ is a good indicator of the rank of $E$, a better indicator for these curves is $\min\{D_E(t) : t \leq 10^6\}$. Thus, if we wanted to use this to find twists of moderate rank, we could compute $\min\{D_{E'}(t) : t \leq 10^6\}$ for many twists $E'$ of $E$ and use the resulting values to pick twists on which to perform the full rank computation.

\begin{remark}
This method is computationally simpler than the method described above since there is no factoring necessary and counting the points on an elliptic curve over a finite field can be done very efficently. Furthermore, since at each step of the computation we are at most incrementing or decrementing the running total, we are able to compute the values in question quickly and without requiring much memory.

On the other hand, this method places an explicit bound on the magnitudes of the squarefree integers by which we are twisting the original curve. In the particular case in which we used it, for example, we only considered twists by integers up to $10^6$. Rogers's method, meanwhile, sets a bound on the heights of the rationals $x$ for which we compute the squarefree parts $D_{f(x)}$ of $f(x)$; while this imposes an implicit bound on the size of the twists considered, it is a much looser one, and twists by squarefree integers of much greater magnitude occur. For instance, note that in our search, in which we considered rationals $x$ of height up to $10^4$, one of the curves of rank 5 with a 21-isogeny was given by twisting by $-531894503$. This explicit restriction is, at least conjecturally, a real disadvantage: since the rank of an elliptic curve $E$ with conductor $N_E$ is conjectured to be bounded by $C \frac{\log N_E}{\log\log N_E}$ for some constant $C$,\footnote{For more details, see \cite[Conjecture 10.5]{ulmer}.} and since the conductor of $E^{(D)}$ is essentially $D^2\cdot N_E$, we want to allow twists by large integers in our search for twists of large rank.
\end{remark}

We take $E$ to be the minimal twist of the single $\overline{\QQ}$-isomorphism class of elliptic curves with 163-isogenies. We compute $\min\{D_{E^{(D)}}(t) : t \leq 10^5\}$ for every squarefree $D$ with $|D| \leq 10^6$ and then compute the ranks of the 100 most promising curves, producing a curve of rank 5, namely $E^{(D)}$ with $D = 376085$.

\begin{table}[h]
$$
\arraycolsep=1.4pt\def\arraystretch{1.2}
\begin{array}{c|c|c|c|c}
j&n&\hbox{$a$-invariants of $E$}&D&\rank_{\QQ}\left(E^{(D)}\right)\\\hline
-2^{18} \cdot 3^3 \cdot 5^3 \cdot 23^3 \cdot 29^3 & 163 & [0, 0, 0, -34790720, 78984748304] & 376085 & {\bf 5} \\
\end{array}
$$
\caption{A new method for the case $n = 163$}
\label{table:163}
\end{table}

\vspace{10pt}

\section{When the genus of $X_0(n)$ is zero.}

When the modular curve $X_0(n)$ has genus zero, we have that $X_0(n)(\QQ)\simeq \PP^1(\QQ)$ and that there are infinitely many rational points corresponding to $\overline{\QQ}$-isomorphism classes of rational elliptic curves with an $n$-isogeny. In this case, the $j$-invariants of these elliptic curves are given as the image of a rational map $j_n:X_0(n)\simeq\PP^1 \to \QQ.$ For example,
$$j_{13}(h) = \frac{(h^2+5h+13)(h^4+7h^3+20h^2+19h+1)^3}{h}.$$

Therefore we need an efficient way to determine for what values of $h\in \QQ$ the minimal twist of the $\overline{\QQ}$-isomorphism class of curves with $j$-invariant $j_n(h)$ is \emph{likely} to have large rank. To do this we will use a heuristic attributed to Nagao which concerns the $L$-function of an elliptic curve.

\subsection{$L$-functions of Elliptic Curves} 
Throughout this section let $E/\QQ$ be an elliptic curve. If $p$ is a prime, we define $N_p$ and $a_p$ as above.   

\begin{definition}
For any prime $p$, the {\it local factor at $p$ of the $L$-series} is defined to be 
$$L_p(T) = \begin{cases}
1-a_pT+pT^2, \hbox{if $E$ has good reduction at $p$},\\
1-T, \hbox{if $E$ has split multiplicative reduction at $p$},\\
1+T, \hbox{if $E$ has non-split multiplicative reduction at $p$},\\
1, \hbox{if $E$ has additive reduction at $p$}.\\
\end{cases}$$
The $L$-function of the elliptic curve $E$ is defined to be 
$$L(E,s) = \prod_{p\geq 2}\frac{1}{L_p(p^{-s})},$$
where the product is taken over all primes $p.$
\end{definition}

The importance of $L$-functions derives principally from the following conjecture, which connects the algebraic structure of an elliptic curve with the behavior of its associated $L$-function, giving an analytic method for computing the rank of the curve.

\begin{conj}[Birch and Swinnerton-Dyer]
Let $E$ be an elliptic curve over $\QQ$, and let $L(E,s)$ be the $L$-function of $E$. Then $L(E,s)$ has a zero at $s=1$ of order equal to $\rank_\QQ(E)$.
\end{conj}

In fact, the Birch and Swinnerton-Dyer conjecture says more than what is stated above. It also gives a formula for the residue of $L(E,s)$ at $s=1$ in terms of various invariants of $E$, but for our purposes this weaker version is sufficient. For more information, the reader is encouraged to see \cite{RubinSilverberg} or \cite{ALR2}.

Assuming this conjecture, we can write $L(E,s) = (s-1)^{R_E} \cdot g(E,s)$, where $g(E,s)$ is some function such that $g(E,1)\neq 0 $ and $R_E = \rank_\QQ(E)$. Computing the log derivative of $L(E,s)$ yields 
$$\frac{L'(E,s)}{L(E,s)} = R_E\cdot\frac{1}{s-1}+h(E,s)$$
where $h(E,s) = g'(E,s)/g(E,s)$ is analytic close to $s=1$. Therefore, 
$$\lim_{s\to 1} \frac{L'(E,s)}{L(E,s)} = R_E\lim_{s\to 1} \frac{1}{s-1} + h(1),$$
and the rank of $E$ can be estimated using the rate at which the limit goes to infinity.

The problem with this estimation is that computing the $L$-function of an elliptic curve is time consuming if it is possible at all. So Nagao defined the following finite product,
$$L_N(E,s) = \prod_{p\leq N} (1-a_pp^{-s}+p^{1-2s})^{-1},$$
as an approximation to the $L$-function $L(E,s)$.

The products $L(E,s)$ and $L_N(E,s)$ have two main differences. The most apparent discrepancy is that $L_N(E,s)$ is truncated to be a finite product rather than an infinite one. The other difference is that $L_N(E,s)$ ignores the deviation between the local factors of primes of good and bad reduction. Treating every prime as if it were a prime of good reduction makes computing the truncated $L$-function for \emph{many} different elliptic curves easier while only changing a finite number of terms in each product. Thus, $L(E,s)$ and $\displaystyle \lim_{N\to\infty} L_N(E,s)$ should still be closely related. 

Next, taking the logarithm of $L_N(E,s)$, defining the resulting sum to be $f_N(E,s)$,
and taking the derivative shows that
$$f_N'(E,s) = - \sum_{p\leq N} \frac{a_p p^{-s}-2p^{1-2s}}{1-a_pp^{-s}+p^{1-2s}}\log p.$$
If we expect $L(E,s)$ and $\displaystyle \lim_{N\to\infty} L_N(E,s)$ to be closely related, we should also expect $\displaystyle \lim_{N\to \infty}f'_N(E,s)$ to be a good approximation to the log derivative of ${L(E,s)}$, and so
$$\lim_{s\to 1}\frac{L'(E,s)}{L(E,s)}\approx \lim_{s\to1}\lim_{N\to\infty}f_N'(E,s).$$
Finally, since our goal is to arrive at a rough heuristic, we switch the limits on the right-hand side without justification to get 
$$\lim_{s\to 1}\frac{L'(E,s)}{L(E,s)}\approx \lim_{N\to\infty}f_N'(E,1)=\lim_{N\to\infty} -\sum_{p\leq N} \frac{a_p p^{-1}-2p^{-1}}{1-a_pp^{-1}+p^{-1}}\log p = \lim_{N\to\infty}\sum_{p\geq N}\frac{2-a_p}{N_p}\log(p).$$
Therefore, if we let 
$$\displaystyle S(E,N) = f_N(E,1) = \sum_{p\geq N}\frac{2-a_p}{N_p}\log(p),$$ we should be able to use $S(E,N)$ for some sufficiently large $N$ to distinguish between elliptic curves of large and small rank. That is to say, when $S(E,N)$ is relatively large, we have reason to suspect that $E$ itself has relatively large rank. 

\begin{remark}
The heuristic we use here was first used by Nagao in \cite{nagao} where he used this heuristic and a construction of Mestre to generate an elliptic curve of rank greater than or equal to 20. It was then used again by Watkins, Donnelly, Elkies, Fisher, Granville, and Rogers as described in \cite{watkins}. One could form many slightly different sums over small primes that all ought to correlate with the rank of a fixed elliptic curve. For a more detailed discussion of these variations and their respective advantages and disadvantages, see \cite[\S 5]{watkins}.
\end{remark}

For each $r\in\QQ$ with height less than some bound $H$, we let $E_{j_n(r)}$ be the minimal model of the elliptic curve with $j$-invariant $j_n(r)$ and use SAGE to compute $S(E_{j_n(r)},N)$ for some fixed ``large'' $N$. Then we use Magma to compute the Selmer ranks and actual ranks of the elliptic curves $E$ for which $S(E,N)$ is largest.

For $n = 13$ and $16$ we let $N = 10^4$ and $H=200$ to get the following table of data.

\begin{table}[h]
$$
\arraycolsep=1.4pt\def\arraystretch{1.2}
\begin{array}{c|c|c|c}
 n & r & \hbox{$a$-invariants} & \rank_\QQ\left(E_{j_n(r)}\right)\\\hline
 13 & 101/60 & [1, 0, 1, -288446980524976098, 59625981683441978523832756] & {\bf 6} \\\hline
 13 & -113/20 & [ 0, 0, 0, -452597990381667, 3706144386347090080926] & 5 \\\hline
 13 & 9/8 & [1, 1, 1, -8657699123, 309987918855281] & 4 \\\hhline{=|=|=|=}
 16 & 187/40 & [1, 1, 1, -30111109596200490, 201112119279475857520645] & 4 \\\hline
 16 & 77/80 & [ 1, 1, 1, -105028852953910, 11925868172151407627315 ] & 3 \\\hline
 16 & 167/114 & [ 1, 0, 0, 21676796948537980, 1144823745894014005574160 ] & 2 \\
\end{array}
$$
\caption{Searching in the genus zero case with the Nagao heuristic}
\label{table:zero_Nagao}
\end{table}

However, a problem arises when we try and use this method for $n= 18$ and $25$. While it is possible to use the heuristic above to compute the elliptic curves that \emph{should} have high rank, the conductors of these curves are so large that Magma is unable to compute their Selmer rank, let alone actual rank, in a reasonable amount of time. In order to find elliptic curves of moderate rank with isogenies of degree 18 or 25, we need a different approach, so we revert to the method of Rogers outlined in Section~\ref{sec-PositiveGenus}. The idea is to use \cite{lmfdb} to find the elliptic curves with minimal conductor among those with an 18-isogeny and among those with a 25-isogeny. We then use Rogers's method to look for twists of these curves that have moderate rank. In fact, we also applied this method to the cases of $n = 13$ and 16 to see how the two methods compare. Below we give a table of the results obtained from these searches.

\begin{table}[h]
$$
\arraycolsep=1.4pt\def\arraystretch{1.2}
\begin{array}{c|c|c|c|c|c}
j&n&\hbox{$a$-invariants of $E$} &H&D&\rank_{\QQ}\left(E^{(D)}\right)\\\hline
-1 \cdot 2^{12} \cdot 7/3 & 13 & [0, -1, 1, -2, -1] & 10^4 & 70557 & 4 \\\hline 
-1 \cdot 3^3 \cdot 7/2 & 13 & [1, -1, 0, -2, 6] & 10^4 & -1049 & 4 \\\hhline{=|=|=|=|=|=}
-1/15 & 16 & [1, 1, 1, 0, 0] & 10^4 &  -1625560485 &  6 \\\hline 
5281^3/(3^4 \cdot 5 \cdot 13) & 16 & [1, 0, 0, -110, 435] & 10^4 & 59243810 & {\bf 7} \\\hhline{=|=|=|=|=|=}
-1 \cdot 5^6 /(2^2\cdot 7) & 18 & [1, 0, 1, -1, 0] & 10^4 &  -10533149 & {\bf 6} \\\hhline{=|=|=|=|=|=}
-1 \cdot 2^12 / 11 & 25 & [0, -1, 1, 0, 0] & 10^4 & -203145767 & {\bf 6} \\\hline 
-1 \cdot 269^3 / (2\cdot 11)& 25  & [1, 1, 1, -28, -69] & 10^4 &  4817182 & {\bf 6} \\
\end{array}
$$
\caption{Twists in the genus zero case}
\label{table:zero_Rogers}
\end{table}

For each value of $n$ we considered---that is, for each $n$ appearing in Theorem~\ref{thm-isog} besides those for which there exist elliptic curves with rational points of order $n$---our searches produced an elliptic curve of rank at least 5 with an $n$-isogeny. We would like to extend this baseline minimum of rank 5 to all possible isogeny degrees, so we now turn our attention briefly to the others, namely $2 \leq n \leq 10$ and $n = 12$. For $2 \leq n \leq 8$, there are known examples of elliptic curves of rank at least 5 that have torsion group isomorphic to $\ZZ/n\ZZ$ (see \cite{Dujella}) and thus have an $n$-isogeny. This leaves the three cases of $n = 9$, $n =10$ and $n = 12$. In each case, the highest known rank of a curve with torsion group isomorphic to $\ZZ/n\ZZ$ is only 4. However, the curve of rank 6 with an 18-isogeny that our search found also has a 9-isogeny,\footnote{To see this, let $E$ be the curve in question and take $P$ to be a generator of the Galois stable cyclic subgroup $G \subseteq E(\overline{\QQ})$ of order 18, and consider the subgroup generated by $2P$.} so we have a curve of rank at least 5 with a 9-isogeny.

For $n = 10$ and $n = 12$, we conducted searches using the method of Rogers, producing the following data.

\begin{table}[h]
$$
\arraycolsep=1.4pt\def\arraystretch{1.2}
\begin{array}{c|c|c|c|c|c}
j&n&\hbox{$a$-invariants of $E$} &H&D&\rank_{\QQ}\left(E^{(D)}\right)\\\hline
179^3 \cdot 2699^3 / (2^2 \cdot 3 \cdot 11^5) & 10 & [1, 0, 0, -10065, -389499] & 10^4 & -802314609 & {\bf 5} \\\hline 
-11^3 \cdot 59^3 / (2^12 \cdot 3 \cdot 5^3) & 12 & [1,0,1,-14,-64] & 10^4 & -148243395 & {\bf 5} \\
\end{array}
$$
\caption{Twists for $n = 10$ and 12}
\label{table:10_and_12}
\end{table}

The data in Tables~\ref{table:positive}, \ref{table:163}, \ref{table:zero_Nagao}, \ref{table:zero_Rogers}, and \ref{table:10_and_12}, together with the rank records recorded by Dujella at \cite{Dujella}, justify the following result.

\begin{theorem}
If $n$ is a positive integer appearing in Theorem \ref{thm-isog}, then there exists an elliptic curve with an $n$-isogeny and rank greater than or equal to 5.
\end{theorem}

\section{Acknowledgements}
The authors would like to thank Filip Najman for suggesting this project and \'Alvaro Lozano-Robledo and Steven J. Miller for helpful conversations throughout the process as well as the referee and editors for their useful comments and a quick editorial process. Both authors would also like to thank the Amherst College Department of Mathematics and Statistics for its support of the undergraduate thesis project during which this work was carried out.


\begin{thebibliography}{9}

\bibitem{Code}
H. Daniels and H. Goodwillie, Magma/SAGE scripts and data related to {\it On the ranks of elliptic curves with isogenies}, available at \url{https://www3.amherst.edu/~hdaniels/rank_data.html}.

\bibitem{Dujella} A. Dujella, ``History of Elliptic Curves Rank Records'', available at \url{https://web.math.pmf.unizg.hr/~duje/tors/rankhist.html}.

\bibitem{ElkiesRR} N. Elkies, {\it $\ZZ^{28}$ in $E(\QQ)$}, Number Theory Listserver, May 2006.

\bibitem{gouvea} F. Gouv\^ea and B. Mazur, {\it The square-free sieve and the rank of elliptic curves}, J. Amer. Math Soc. {\bf 4}:1 (1991), 1--23.

\bibitem{kenku2} M. A. Kenku, {\it The modular curve $X_0(39)$ and rational isogeny}, Math. Proc. Cambridge Philos. Soc. \textbf{85} (1979), 21--23.

\bibitem{kenku3} M. A. Kenku, {\it The modular curves $X_0(65)$ and $X_0(91)$ and rational isogeny}, Math. Proc. Cambridge Philos. Soc. \textbf{87} (1980), 15--20.

\bibitem{kenku4} M. A. Kenku, {\it The modular curve $X_0(169)$ and rational isogeny}, J. London Math. Soc. (2) \textbf{22} (1980), 239--244.

\bibitem{kenku5} M. A. Kenku, {\it The modular curve $X_0(125)$, $X_1(25)$ and $X_1(49)$}, J. London Math. Soc. (2) \textbf{23} (1981), 415--427.

\bibitem{lmfdb} LMFDB Collaboration, {\it The L-functions and modular forms database}, available at \url{http://www.lmfdb.org}.

\bibitem{ALR1} \'A. Lozano-Robledo, {\it On the field of definition of p-torsion points on elliptic curves over the rationals,} Mathematische Annalen, Vol 357, Issue 1 (2013), 279--305.

\bibitem{ALR2} \'A. Lozano-Robledo, {\it Elliptic Curves, Modular Forms, and Their $L$-functions}, American Mathematical Society, Providence, Rhode Island, 2010.

\bibitem{magma} W. Bosma, J. Cannon, and C. Playoust, {\it The {M}agma algebra system. {I}. {T}he user language}, J. Symbolic Comput., \textbf{24} (1997), 235--265.

\bibitem{mazur1} B. Mazur, {\it Rational isogenies of prime degree}, Invent. Math. \textbf{44} (1978), 129--162.

\bibitem{mazur2} B. Mazur, {\it Finding meanings and error terms}, Bull. Amer. Math. Soc. \textbf{45} (2008), no. 2, 185--228.

\bibitem{mordell} L.~J. Mordell, {\it On the rational solutions of the indeterminate equations of the third and fourth degrees}, Proc. Cambridge Philos. Soc. \textbf{21} (1922), 179--192.

\bibitem{nagao} K. Nagao, {\it An Example of Elliptic Curve over Q with Rank $\geq$ 20}, Proc. Japan Acad., \textbf{69} (1993), 291--293.

\bibitem{PPVW} J. Park, B. Poonen, J. Voight, and M. Wood {\it A heuristic for boundedness of ranks of elliptic curves}, available at \url{http://arxiv.org/abs/1602.01431}.

\bibitem{rogers} N.~F. Rogers, {\it Rank computations for the congruent number elliptic curves}, Experiment. Math. 9, no. 4 (2000), 591--594.

\bibitem{RubinSilverberg} K. Rubin and A. Silverberg, {\it Ranks of elliptic curves}, Bull. Amer. Math. Soc. {\bf 39} (2002), 455--474.

\bibitem{sage} Sage Mathematics Software (Version 6.10.0), The Sage Developers, 2015, \url{http://www.sagemath.org}.

\bibitem{silverberg} A. Silverberg, {\it The distribution of ranks in families of quadratic twists of elliptic curves}, Ranks of elliptic curves and random matrix theory, 171-–176, London Math. Soc. Lecture Note Ser., 341, Cambridge Univ. Press, Cambridge, 2007. 

\bibitem{silverman} J. H. Silverman, {\it The arithmetic of elliptic curves}, Springer-Verlag, 2nd Edition, New York, 2009.

\bibitem{ulmer} D. Ulmer, {\it Elliptic curves with large rank over function fields}, Annals of Mathematics, {\bf 155} (2002), 295 -- 315


\bibitem{watkins} M. Watkins, S. Donnelly, N.~D. Elkies, T. Fisher, A. Granville, and N.~F. Rogers, {\it Ranks of quadratic twists of elliptic curves}, Num\'ero consacr\'e au trimestre "M\'ethodes arithm\'etiques et applications'', automne 2013, 63--98, Publ. Math. Besan\c{c}on Alg\`{e}bre Th\'eorie Nr., 2014/2, Presses Univ. Franche-Comt\'{e}, Besan\c{c}on, 2015.

\bibitem{weil} A. Weil, {\it L'arithm\'etique sur les courbes alg\'ebriques}, Acta Math. \textbf{52} (1929) 281--315.

\end{thebibliography}
\end{document}